\documentclass[preprint,3pt]{elsarticle}

\usepackage{amssymb}
\usepackage{amsthm,amssymb,amsmath}
\usepackage{pdfpages}
\usepackage{graphicx}
\usepackage{geometry}
\usepackage{url}
\usepackage{ulem}

\theoremstyle{plain} \newtheorem{thm}{\bf Theorem}[section]
 \newtheorem{lem}[thm]{\bf Lemma}
\newtheorem{proposition}[thm]{\bf Proposition} \newtheorem{defn}[thm]{\bf Definition} \theoremstyle{remark} \newtheorem{remark}[thm]{\bf Remark}

\newcommand{\Rg}       {{\hbox{I\kern-.22em\hbox{R}}}}
\newcommand{\Pg}       {{\hbox{I\kern-.22em\hbox{P}}}}
\newcommand{\Eg}       {{\hbox{I\kern-.22em\hbox{E}}}}

\journal{Insurance: Mathematics and Economics}

\begin{document}

\begin{frontmatter}



\title{Non-parametric threshold estimation for classical risk process perturbed by diffusion}


\author{Chunhao Cai, Junyi Guo and Honglong You}

\address{School of mathematical science\\
Nankai University\\youhonglong815@163.com}

\begin{abstract}
In this paper,we consider a macro approximation of the flow of a risk reserve, which is first introduced by \cite{Shimizu2009}. The process is observed at discrete time points. Because we cannot directly observe each jump time and size then we will make use of a technique for identifying the times when jumps larger than a suitably defined threshold occurred. We estimate the jump size and survival probability of our risk process from discrete observations.
\end{abstract}


\end{frontmatter}

\section{Introduction}
In the field of financial and insurance mathematics, the following model is the most commonly used:
\begin{equation}\label{1}
X_t=x+ct+\sigma W_{t}-\sum_{i=1}^{N_{t}}\gamma_{i},
\end{equation}
where $x$ is the initial value and $c>0$ is a premium rate, $\sigma>0$ represents the diffusion volatility, $N_{t}$ is a Poisson process with rate $\lambda$ counting the number of jumps up to time $t>0$, $\gamma_{1},\gamma_{2},...$ are independent and identically distributed $(i.i.d.)$ positive random variables with distribution $F$, $\sum_{i=1}^{N_{t}}\gamma_{i}$ represents the aggregate jumps up to time $t$, $W_{t}$ is a Wiener process independent of $N_{t}$ and $\gamma_{i}$. In order to avoid the almost sure ruin of \eqref{1}, we assume that the premium rate can be written as $c=(1+\theta_{0})\lambda\mu$ where the premium loading factor $\theta_{0}$ is positive and $\mu=\int_{0}^{\infty}uF(du)$.

This model was first introduced by \cite{Gerber(1970)} and further studied by many authors during the last few years (\cite{Dufresne and Gerber(1991),Furrer and Schimidli(1994), Veraverbeke(1993),Schimidli(1995)}). Many results on ruin probability and other ruin problems have been obtained by the works mentioned above. If we define the survival probability of \eqref{1} as $\Phi(x)$ then the problem is that, apart from some special cases, a general expression for $\Phi(x)$ does not exist. So in this paper we will construct an estimator of $\Phi(x)$ without specifying any parametric model on  $F$, the distribution function of $\gamma_i$.

In our model we will assume that $\sigma$ and $\lambda$ are unknown parameters. In the classical Poisson risk model, the estimation of ruin probability has been considered by many authors (\cite{Frees(1986), Hipp(1989), Croux, Pitts,Bening, Politis, Mnatsakanov}). Not the same observation rules as in the work of \cite{Shimizu},  we only consider a discrete record of $n+1$ observations $\{X_{t^{n}_{0}},X_{t^{n}_{1}},...,X_{t^{n}_{n-1}},X_{t^{n}_{n}}\}$ where  $t^{n}_{i}=ih_{n}$, $T_{n}=t^{n}_{n}$ and $h_{n}>0$. This observation is also important since the real data is always obtained at discrete time points.

When we want to estimate the distribution of $\gamma_i$, we have to separate the contributions of the diffusion part with respect to the jump of the poisson process so the threshold estimation which has been introduced in (\cite{Mancini(2004), Shimizu2009, Shimizu2009b, Shimizu and Yoshida(2006), Shimizu(2006)})  can be used. We only accept there is a jump for the  Poisson process between the interval  $(t^{n}_{i-1},t^{n}_{i}]$ if and only if the increment $\Delta_{i} X=X_{t^{n}_{i}}-X_{t^{n}_{i-1}}$ has a too big absolute value. We define this threshold function $\vartheta(h_{n})$ which goes to zero when $h_n\rightarrow 0$.  Recall that the estimator of the empirical distribution  function of $\gamma_i$ can be written by
$$
\hat{F}_{N_{t}}(u)=\frac{1}{N_{t}}\sum_{i=1}^{N_{t}}I_{\{\gamma_i\leq u\}},
$$
then we can try to write the estimator as
\begin{equation}\label{distribution estimator}
\hat{F}_n(u)=\frac{1}{\sum_{i=1}^nI_{\{|\Delta_iX|>\vartheta(h_n)\}}}\sum_{i=1}^nI_{\{ I_{\{|\Delta_iX|>\vartheta(h_n)\}}\Delta_iX\leq u\}}.
\end{equation}

In this paper we will try to study the properties of the estimator defined in \eqref{distribution estimator}, such as the strong consistency, asymptotic normality. We also obtain the weak consistency in a sense of the integrated squared error $(ISE)$ of the estimator of survival probability such as in \cite{Mnatsakanov}.

The rest of the paper is organized as follows, in section \ref{pre} we will give some notations and construct the estimators for all the unknown parameters and the distribution of $\gamma_i$.  In section \ref{Main}, we will give the asymptotic properties of the estimators given in section \ref{pre}.
In section \ref{mMain}, we obtain the weak consistency in a sense of $ISE$ of the estimator of survival probability of our risk process.
\section{Preliminaries}\label{pre}
\subsection{General Notation}\label{GN}
Throughout the paper, we use the primary notations and assumptions.

$\bullet\quad $ symbols $\xrightarrow{P}$ and $\xrightarrow{D}$ stand for the convergence in probability, and in law, respectively;

$\bullet\quad $ $L_{F}$(respectively, $l_{F}$) is the Laplace(respectively, Stieltjes) transform for a function $F$: for $s>0$
$$L_{F}(s)=\int_{0}^{\infty}e^{-su}F(u)du;\quad l_{F}(s)=\int_{0}^{\infty}e^{-su}F(du)$$
where in $L_F$, $F$ is for every function and in $l_F$, $F$ is the distribution function.

$\bullet\quad $ $\|f\|_{K}:=(\int_{0}^{K}|f(t)|^{2}dt)^{\frac{1}{2}}$ for a function $f$. In particular, $\|f\|=\|f\|_{\infty}$. We say that $f\in L^{2}(0,K)$ if $\|f\|_{K}<\infty$ and $f\in L^{2}(0,\infty)$ if $\|f\|<\infty$, respectively.

$\bullet\quad $ For a stochastic sequence $X_{n}$, we denote by $X_{n}=O_{P}(R_{n})$ if $\frac{X_{n}}{R_{n}}$ is bounded in probability and $X_{n}=o_{P}(R_{n})$ if $\frac{X_{n}}{R_{n}}\xrightarrow{P}0$

$\bullet\quad $ $\mathbb{E}$ is a compact subset of $(0,\infty)$.

$\bullet\quad $ $k_{n}>0$ is a real-value sequence.\\

We make the following assumptions.\\

$\mathbf{A}.$
Let $\sigma$, $\gamma$ satisfy: $\sigma<Q$, $0<\Gamma\leq\gamma<Q$, with $Q>0$.\\

$\mathbf{B}.$
For $\delta\in(0,\frac{1}{2})$, $p,q>1$ with $p^{-1}+q^{-1}=1$ and $s\in \mathbb{E}$,
\\
\\

1. $\alpha_{n}^{\delta}(k_{n}^{2})\rightarrow0$ as $n\rightarrow\infty$;
\\
\\

2. $\sqrt{T_{n}}(\alpha_{n}^{\delta}(k_{n})+\omega_{n}(s,p,q)+\Gamma(s,\varphi_{n}))\rightarrow0$ as $n\rightarrow\infty$,
\\
\\
where
$$\alpha_{n}^{\delta}(k_{n})=k_{n}h_{n}^{\frac{1}{2}-\delta},\quad\Gamma(s,\varphi_{n})=\lambda\int_{0}^{\infty}(M_{s}-\varphi_{n}\circ M_{s})(x)F(dx)$$
 and
$$\omega_{n}(s,p,\rho)=h_{n}^{\frac{1+\rho}{q}-1}(\lambda\int_{0}^{\infty}|M_{s}(x)|^{p}F(dx))^{\frac{1}{p}}.$$

\subsection{Estimator of unknown parameters and $l_{F}$}\label{12}
Now we will try to construct the estimator of $\sigma^{2}$, $\lambda$, $\rho=\frac{\lambda\mu}{c}$ and $l_{F}$. Because we only consider the case where observations are discrete, that is to say the jumps are not observable. To overcome the problem, \cite{Mancini(2004)}, \cite{Shimizu and Yoshida(2006)} and \cite{Shimizu2009b} have used a jump-discriminant filter of the form
\begin{eqnarray*}
\mathcal{C}_{i}^{n}(\vartheta(h_{n}))=\{\omega\in\Omega;|\Delta_{i}X|>\vartheta(h_{n})\}
\end{eqnarray*}
to discriminate between jumps and large Brownian shocks in an interval $(t_{i-1}^{n},t_{i}^{n}]$ in the cases of jump-diffusions. They judged that no jump had occurred if $|\Delta_{i}X|\leq\vartheta(h_{n})$ and that a single jump had occurred if $|\Delta_{i}X|>\vartheta(h_{n})$ by choosing the
threshold $\vartheta(h_{n})$ suitably.

According to the work of \cite{Shimizu and Yoshida(2006)}, we can define the following estimators of $\sigma^{2}$, $\lambda$, $\rho$ and $l_{F}$ :
$$\widetilde{\sigma_{n}}^{2}=\frac{\sum_{i=1}^{n}|\Delta_{i}X-ch_{n}|^{2}I_{\{|\Delta_{i}X|\leq\vartheta(h_{n})\}}}{h_{n}\sum_{i=1}^{n}I_{\{|\Delta_{i}X|\leq\vartheta(h_{n})\}}},
\quad\widetilde{\lambda_{n}}=\frac{\sum_{i=1}^{n}I_{\{|\Delta_{i}X|>\vartheta(h_{n})\}}}{T_{n}},$$
$$\widetilde{\rho_{n}}=\frac{1}{c}\frac{\sum_{i=1}^{n}(|\Delta_{i}X|) I_{\{|\Delta_{i}X|>\vartheta(h_{n})\}}}{T_{n}},\quad\widetilde{l_{F}}_{n}(s)=\frac{\sum_{i=1}^{n}( e^{-s|\Delta_{i}X|})I_{\{|\Delta_{i}X|>\vartheta(h_{n})\}}}{\sum_{i=1}^{n}I_{\{|\Delta_{i}X|>\vartheta(h_{n})\}}}$$
where $\vartheta(h_{n})=Lh_{n}^{\omega}$, $L>0$, $\omega\in(0,\frac{1}{2})$ and
$s\in\mathbb{E}$.

\begin{remark}
Here, we choose $\vartheta(h_n)=Lh_n^{\omega}$ for a constant $L>0$ and $\omega\in (0,\frac{1}{2})$. In \cite{Mancini(2004)}, the author proposed the jump-discriminant threshold function such as $\vartheta(h_n)=L\sqrt{h_n\log \frac{1}{h_n}}$ for a constant $L>0$. In \cite{Shimizu and Yoshida(2006)}, the author proposed $\vartheta(h_n)=Lh_n^\omega$ for a constant $L>0$ and $\omega\in (0,\,\frac{1}{2})$. First of all, it is obvious that their threshold functions satisfy the conditions: $\vartheta(h_n)\rightarrow 0$ and $\sqrt{h_n}\vartheta(h_n)^{-1}\rightarrow 0$. If $\vartheta(h_n)$ does not converge to zero then one cannot detect small jumps even when $n\rightarrow \infty$. Moreover, if $\vartheta(h_n)$ converges to zero faster than the order of $\sqrt{h_n}$ then the filter $\mathcal{C}_i^n(\vartheta(h_n))$ would possibly misjudge the Brownian noise as a small jump sine the variation of the Brownian motion is of order $\sqrt{h_n\log \frac{1}{h_n}}$. Actually, the threshold function $\vartheta(h_n)$ in \cite{Shimizu and Yoshida(2006)} and \cite{Mancini(2004)} are all suited for our aim. For the convenience of proof on our results, we choose $\vartheta(h_n)$ in \cite{Shimizu and Yoshida(2006)}. For a detailed account on the threshold functions we refer to \cite{Shimizu2009b}.
\end{remark}
\section{Properties of the Estimators}\label{Main}
\subsection{Asymptotic Normality of $F_n(u)$}
First we will  try to get the asymptotic normality of $\hat{F}_n(u)$. Before that we will define $\bar{F}=1-F$ where $F$ is the distribution of $\gamma_i$ and $\hat{\bar{F}}_n(u)=1-\hat{F}_n(u)$. Then
$$
\hat{\bar{F}}_n(u)=\frac{1}{\sum_{i=1}^{n}I_{\{|\Delta_{i}X|>\vartheta(h_{n})\}}}\sum_{i=1}^{n}I_{\{\Delta_{i}XI_{\{|\Delta_{i}X|>\vartheta(h_{n})\}}> u\}},$$
then we have the following result:
\begin{thm}\label{aympto hat  bar F}
Suppose that $nh_n\rightarrow \infty$, $h_n\rightarrow 0$ as $n\rightarrow \infty$ and the condition $\mathbf{A}$ in section \ref{pre} is satisfied, then
\begin{equation}\label{asym hat bar F}
\sqrt{T_{n}}\left(\hat{\bar{F}}_n(u)-\overline{F}(u)\right)\xrightarrow{D}\mathbb{N}(0,\frac{\overline{F}(u)(1-\overline{F}(u))}{\lambda}).
\end{equation}
then we change to $\bar{F}_n(u)$ and $F(u)$ that is
\begin{equation}
\sqrt{T_{n}}\left(\hat{F}_n(u)-F(u)\right)\xrightarrow{D}\mathbb{N}(0,\frac{F(u)(1-F(u))}{\lambda})
\end{equation}
\end{thm}
To prove this theorem, we need the following lemma:
\begin{proposition}\label{ccite}
Following from the condition of Theorem \ref{aympto hat  bar F}, for any $\epsilon>0$,
\begin{eqnarray}\label{zzh}
\lim_{n\rightarrow\infty}P\left(|\Delta_{i}XI_{\{|\Delta_{i}X|>\vartheta(h_{n})\}}-\gamma_{\tau^{(i)}}I_{\{\Delta_{i}N\geq1\}}|>\epsilon\right)=0,
\end{eqnarray}
where $\gamma_{\tau^{(i)}}$ is the eventual jump in time interval $(t^{n}_{i-1},t^{n}_{i}]$.
\end{proposition}
\begin{proof}
\begin{eqnarray}\label{zzzh}
&&P\left(|\Delta_{i}XI_{\{|\Delta_{i}X|>\vartheta(h_{n})\}}-\gamma_{\tau^{(i)}}I_{\{\Delta_{i}N\geq1\}}|>\epsilon\right)\nonumber\\&&\leq
P\left(|\Delta_{i}XI_{\{|\Delta_{i}X|>\vartheta(h_{n}),\Delta_{i}N=0\}}|>\frac{\epsilon}{3}\right)
\nonumber\\&&+P\left(|\Delta_{i}XI_{\{|\Delta_{i}X|>\vartheta(h_{n})\}}-\gamma_{\tau^{(i)}}|I_{\{\Delta_{i}N=1\}}>\frac{\epsilon}{3}\right)
\nonumber\\&&+P\left(|\Delta_{i}XI_{\{|\Delta_{i}X|>\vartheta(h_{n})\}}-\gamma_{\tau^{(i)}}|I_{\{\Delta_{i}N\geq2\}}>\frac{\epsilon}{3}\right).
\end{eqnarray}
Due to $P\{|\sigma \int_{t_{i-1}^{n}}^{t_{i}^{n}}dW_{s}|\geq c\}\leq2e^{\frac{-c^{2}}{2Q^{2}h_{n}}}$, the first term in \eqref{zzzh} is dominated by
\begin{eqnarray*}\label{ddh}
&&P\left(ch_{n}>\frac{\epsilon}{6}\right)+P\left(|\sigma\int_{t_{i-1}^{n}}^{t_{i}^{n}}dW_{s}|>\frac{\epsilon}{6}\right)
\\&&\leq P\left(ch_{n}>\frac{\epsilon}{6}\right)+2e^{\frac{-\epsilon^{2}}{72Q^{2}h_{n}}}.
\end{eqnarray*}
The therm
$$|\Delta_{i}XI_{\{|\Delta_{i}X|>\vartheta(h_{n})\}}-\gamma_{\tau^{(i)}}|I_{\{\Delta_{i}N=1\}}$$
is equal to
\begin{eqnarray*}\label{zxzzh}
&&\left||ch_{n}+\sigma\int_{t_{i-1}^{n}}^{t_{i}^{n}}dW_{s}+\gamma_{\tau^{(i)}}|I_{\{|\Delta_{i}X|>\vartheta(h_{n})\}}-\gamma_{\tau^{(i)}}\right|I_{\{\Delta_{i}N=1\}}
\\&&\leq|ch_{n}+\sigma\int_{t_{i-1}^{n}}^{t_{i}^{n}}dW_{s}|+|\gamma_{\tau^{(i)}}|I_{\{|\Delta_{i}X|\leq\vartheta(h_{n}),\Delta_{i}N=1\}}.
\end{eqnarray*}
Thus, \eqref{zzzh} is dominated by
\begin{eqnarray}\label{ddf}
2P\left(ch_{n}>\frac{\epsilon}{6}\right)+4e^{\frac{-\epsilon^{2}}{72Q^{2}h_{n}}}+P\left(\Delta_{i}N=1\right)+P\left(\Delta_{i}N\geq2\right).
\end{eqnarray}
Therefor, \eqref{zzzh} tends to zero as $n\rightarrow\infty$.
\end{proof}

Now we will prove the Theorem \ref{aympto hat  bar F}. Because
\begin{eqnarray}\label{asgh}
&&\sqrt{T_{n}}\left[\frac{1}{\sum_{i=1}^{n}I_{\{|\Delta_{i}X|>\vartheta(h_{n})\}}}\sum_{i=1}^{n}I_{\{\Delta_{i}XI_{\{|\Delta_{i}X|>\vartheta(h_{n})\}}> u\}}-\overline{F}(u)\right]
\nonumber\\&&=\sqrt{T_{n}}\left[\frac{\frac{\sum_{i=1}^{n}\left(I_{\{\Delta_{i}XI_{\{|\Delta_{i}X|>\vartheta(h_{n})\}}> u\}}-\overline{F}(u)I_{\{|\Delta_{i}X|>\vartheta(h_{n})\}}\right)}{T_{n}}}{\frac{\sum_{i=1}^{n}I_{\{|\Delta_{i}X|>\vartheta(h_{n})\}}}{T_{n}}}\right]
\nonumber\\&&=\frac{J}{G}.
\end{eqnarray}
where
$$
J=\frac{\sum_{i=1}^{n}\left(I_{\{\Delta_{i}XI_{\{|\Delta_{i}X|>\vartheta(h_{n})\}}> u\}}-\overline{F}(u)I_{\{|\Delta_{i}X|>\vartheta(h_{n})\}}\right)}{\sqrt{T_{n}}}
$$
and
$$
G=\frac{\sum_{i=1}^{n}I_{\{|\Delta_{i}X|>\vartheta(h_{n})\}}}{T_{n}}
$$
First, we will calculate the limit of the expectation of $J$:

\begin{eqnarray*}\label{adgh}
&&\lim_{n\rightarrow\infty}E[J]\nonumber\\&&=\lim_{n\rightarrow\infty}\frac{n}{\sqrt{T_{n}}}\left[E(I_{\{\Delta_{i}XI_{\{|\Delta_{i}X|>\vartheta(h_{n})\}}> u\}}-\overline{F}(u)I_{\{|\Delta_{i}X|>\vartheta(h_{n})\}})\right]
\nonumber\\&&=\lim_{n\rightarrow\infty}\frac{n}{\sqrt{T_{n}}}E(H)
\nonumber\\&&=\lim_{n\rightarrow\infty}\frac{n}{\sqrt{T_{n}}}\left[P(\Delta_{i}XI_{\{|\Delta_{i}X|>\vartheta(h_{n})\}}> u)-\overline{F}(u)P(|\Delta_{i}X|>\vartheta(h_{n}))\right],
\end{eqnarray*}
where, $H=I_{\{\Delta_{i}XI_{\{|\Delta_{i}X|>\vartheta(h_{n})\}}> u\}}-\overline{F}(u)I_{\{|\Delta_{i}X|>\vartheta(h_{n})\}}$.

By Proposition ~\ref{ccite} and \cite{Mancini(2004)}, we have
\begin{eqnarray}\label{aadgh}
&&\lim_{n\rightarrow\infty}P(\Delta_{i}XI_{\{|\Delta_{i}X|>\vartheta(h_{n})\}}> u)\nonumber\\&&=P(\gamma_{\tau^{(i)}}I_{\{\Delta_{i}N\geq1\}}> u)
\nonumber\\&&=P(\gamma_{\tau^{(i)}}I_{\{\Delta_{i}N\geq1\}}> u|\Delta_{i}N\geq1)P(\Delta_{i}N\geq1)
\nonumber\\&&=P(\gamma_{\tau^{(i)}}> u|\Delta_{i}N\geq1)P(\Delta_{i}N\geq1)
\nonumber\\&&=P(\gamma_{\tau^{(i)}}> u)P(\Delta_{i}N\geq1)
\nonumber\\&&=\overline{F}(u)(\lambda h_{n}+o( h_{n}))
\end{eqnarray}
and
\begin{eqnarray}\label{aadgbh}
\lim_{n\rightarrow\infty}P(|\Delta_{i}X|>\vartheta(h_{n}))&=&P(\Delta_{i}N\geq1)\nonumber\\&=&\lambda h_{n}+o( h_{n}).
\end{eqnarray}
Therefor,
$$
\lim_{n\rightarrow\infty}E[J]=0.
$$
Now, we will calculate the limit of the variation of $J$:
\begin{eqnarray*}\label{aadgh}
&&\lim_{n\rightarrow\infty}Var[J]\nonumber\\&&=\lim_{n\rightarrow\infty}\frac{n}{\sqrt{T_{n}}}\left[Var(I_{\{\Delta_{i}XI_{\{|\Delta_{i}X|>\vartheta(h_{n})\}}> u\}}-\overline{F}(u)I_{\{|\Delta_{i}X|>\vartheta(h_{n})\}})\right]
\nonumber\\&&=\lim_{n\rightarrow\infty}\frac{n}{T_{n}}Var(H)
\nonumber\\&&=\lim_{n\rightarrow\infty}\frac{n}{T_{n}}[E(H)^{2}-E^{2}(H)]
\nonumber\\&&=\lim_{n\rightarrow\infty}\frac{n}{T_{n}}[E(H)^{2}].
\end{eqnarray*}
The term $E(H)^{2}$ is equal to
\begin{eqnarray*}\label{aadgh}
E(I^{2}_{\{\Delta_{i}XI_{\{|\Delta_{i}X|>\vartheta(h_{n})\}}> u\}}&+&\overline{F}(u)^{2}I^{2}_{\{|\Delta_{i}X|>\vartheta(h_{n})\}}\\&-&2\overline{F}(u)I_{\{\Delta_{i}XI_{\{|\Delta_{i}X|>\vartheta(h_{n})\}}> u\}}I_{\{|\Delta_{i}X|>\vartheta(h_{n})\}}).
\end{eqnarray*}
Due to $I^{2}=I$ and \\$I_{\{\Delta_{i}XI_{\{|\Delta_{i}X|>\vartheta(h_{n})\}}> u\}}I_{\{|\Delta_{i}X|>\vartheta(h_{n})\}}=I_{\{\Delta_{i}XI_{\{|\Delta_{i}X|>\vartheta(h_{n})\}}> u\}}$, $E(H)^{2}$ is equal to
\begin{eqnarray*}\label{ajdgh}
E(I_{\{\Delta_{i}XI_{\{|\Delta_{i}X|>\vartheta(h_{n})\}}> u\}}&+&\overline{F}(u)^{2}I_{\{|\Delta_{i}X|>\vartheta(h_{n})\}}\\&-&2\overline{F}(u)I_{\{\Delta_{i}XI_{\{|\Delta_{i}X|>\vartheta(h_{n})\}}> u\}}).
\end{eqnarray*}
By \eqref{aadgh} and \eqref{aadgbh},
\begin{eqnarray*}\label{ajdgh}
\lim_{n\rightarrow\infty}Var[J]&=&\lim_{n\rightarrow\infty}\frac{n}{T_{n}}[\lambda h_{n}\overline{F}(u)(1-\overline{F}(u))+o(h_{n})]
\\&=&\lambda \overline{F}(u)(1-\overline{F}(u)).
\end{eqnarray*}
Applying the central limit theorem, we have
$$
J\xrightarrow{D}\mathbb{N}(0,\lambda \overline{F}(u)(1-\overline{F}(u)))
$$
as  $n\rightarrow\infty$

From the Lemma ~\ref{cite}, we have $G\xrightarrow{P}\lambda$ and then by the Slutsky's theorem, we have
$$\frac{J}{G}\rightarrow^{D}\mathbb{N}(0,\frac{\overline{F}(u)(1-\overline{F}(u))}{\lambda})$$
\subsection{Study of the estimatiors in section \ref{12}}
As is well known in the inference for discretely observed diffusions that:
$$
\widetilde{\sigma_{n}^{2}}\xrightarrow{P}\sigma^{2},\quad n\rightarrow\infty.
$$
In addition, if $nh_{n}^{2}\rightarrow0$,
$$\sqrt{n}(\widetilde{\sigma^{2}_{n}}-\sigma^{2})\xrightarrow{D}N(0,2\sigma^{4})$$
as $n\rightarrow\infty$.
The proof is from Theorem 3.1 in \cite{Shimizu2009}.

In order to find the asymptotic nomality of  $\widetilde{\lambda_{n}}$, $\widetilde{\rho_{n}}$  and $\widetilde{l_{F_n}}$,  We will first define a truncation function $\varphi_{n}(x)$ satisfying the following conditions:

1. $|\varphi_{n}(x)|\leq\beta_{n}$ $a.e.$ for a real-valued sequence $\beta_{n}$ such that $\beta_{n}\uparrow\infty$ as $n\rightarrow\infty$.

2. $\varphi_{n}(x)\rightarrow x$ $a.e.$ as $n\rightarrow\infty$.

Through this $\varphi_n$ we will redefine the estimator of $\lambda$, $\rho$ and $l_F$. Moreover, we consider the estimator of the product of $\lambda l_F$. They are defined as
\begin{eqnarray*}
\widetilde{\rho_{n}}^{*}&=&\frac{1}{c}\frac{\sum_{i=1}^{n}(\varphi_{n}\circ M_{s}^{1})(|\Delta_{i}X|) I_{\{|\Delta_{i}X|>\vartheta(h_{n})\}}}{T_{n}},
\end{eqnarray*}
\begin{eqnarray*}
\widetilde{\lambda}\widetilde{l_{F}}_{n}^{*}(s)&=&\frac{\sum_{i=1}^{n}(\varphi_{n}\circ  M_{s}^{2})(|\Delta_{i}X|) I_{\{|\Delta_{i}X|>\vartheta(h_{n})\}}}{T_{n}},
\end{eqnarray*}
where $M_{s}^{1}(x)=x, M_{s}^{2}(x)=e^{-sx}$,
\begin{eqnarray*}
 \varphi_{n}\circ M_{s}(x)= \left\{\begin{array}{l}
M_{s}(x)  \quad if\left(M_{s}(x)\bigvee\sup_{s\in\mathbb{E}}(M_{s}(x))'_{s}\bigvee\sup_{s\in\mathbb{E}}(M_{s}(x))'_{x}\right)\leq \kappa_{n},\\
 0\quad \quad \quad Otherwise,
\end{array}\right\}
\end{eqnarray*}
and $\kappa_{n}>0$ is a real-value sequence, we have the following results:
\begin{lem}\label{cite}
Supposes $nh_{n}\rightarrow\infty$, $h_{n}\rightarrow0$ as $n\rightarrow\infty$,
then
\begin{equation}\label{conv lambda}
\widetilde{\lambda_{n}}\xrightarrow{P} \lambda,\quad \widetilde{\rho_{n}}^{*}\xrightarrow{P} \rho,
\end{equation}
and
\begin{equation}\label{conv lambda F}
\sup_{\{s|s\in \mathbb{E}\}}|\widetilde{\lambda_{n}}\widetilde{l_{F}}_{n}^{*}(s)-\lambda l_{F}(s)|\xrightarrow{P} 0.
\end{equation}
In addition, if the Condition $\mathbf{B}$ in section \ref{GN} is satisfied, $\int_{0}^{\infty}x^{2}F(dx)<\infty$ and $nh_{n}^{1+\beta}\rightarrow0$ for some $\beta\in(0,1)$, we have
\begin{equation}\label{asym lambda}
\sqrt{T_{n}}(\widetilde{\lambda_{n}}-\lambda)\xrightarrow{D}N(0,\lambda),
\end{equation}
\begin{equation}\label{asym rho}
\sqrt{T_{n}}(\widetilde{\rho_{n}}^{*}-\rho)\xrightarrow{D}N(0,\frac{\lambda}{c^{2}}\int_{0}^{\infty}x^{2}F(dx))
\end{equation}
and
\begin{equation}\label{asym lambda F}
\sqrt{T_{n}}\left(\widetilde{\lambda_{n}}\widetilde{l_{F}}_{n}^{*}(s)-\lambda l_{F}(s)\right)\xrightarrow{D}N(0,\lambda\int_{0}^{\infty}e^{-2sx}F(dx)),
\end{equation}
as $n\rightarrow\infty$.
\end{lem}
\begin{proof}
First from the theorem 3.2 and 3.4 in \cite{Shimizu2009}, it is easy to get \eqref{conv lambda} and \eqref{conv lambda F}. So we only need to check \eqref{asym lambda},\eqref{asym rho} and \eqref{asym lambda F}. We define
\begin{equation*}\label{nu}
\nu_n(s)=\frac{\sum_{i=1}^n\varphi_n \circ M_s(|\Delta_iX|)I_{\{|\Delta_iX|>\vartheta(h_n)\}}}{T_n}
\end{equation*}
and
\begin{equation*}\label{nu s}
\nu(s)=\lambda \int_0^{\infty}M_s(x)F(dx)
\end{equation*}
Notice that
\begin{equation*}
\sqrt{T_n}(\nu_n(s)-\nu(s))=\sum_{i=1}^nY_{n,i}
\end{equation*}
where
\begin{equation}
Y_{n,i}=\frac{1}{\sqrt{T_n}}\left(\varphi_n\circ M_s(|\Delta_iX|)I_{\{|\Delta_iX|\vartheta(h_n)\}}-h_n\lambda\int_0^{\infty}M_s(x)F(dx)\right)
\end{equation}
we can apply the central limit theorem for $\{Y_{n,i}\}_{1\leq i\leq n}$. If the following conditions \eqref{con CLM 1}-\eqref{con CLM 3} are satisfied, we can obtain \eqref{asym lambda}-\eqref{asym lambda F} by \cite{Shiryaev}
\begin{equation}\label{con CLM 1}
\sum_{i=1}^n|E[Y_{n,i}]|\xrightarrow{P}0,
\end{equation}
\begin{equation}\label{con CLM 2}
\sum_{i=1}^nE[Y_{n,i}]^2\xrightarrow{P}\lambda \int_0^{\infty}M_s^2F(dx),
\end{equation}
\begin{equation}\label{con CLM 3}
\sum_{i=1}^nE[Y_{n,i}]^4\xrightarrow{P}0
\end{equation}
The details of the proofs of \eqref{con CLM 1}-\eqref{con CLM 3} may refer to Theorem \cite{Shimizu2009b}.
\end{proof}

\section{Integrated Squared Error (ISE) of the Esitimator}\label{mMain}
\subsection{Weak consistency in the $\mathbf{ISE}$ sense}
We will present our important result that states a convergence in probability of the $ISE$ of the estimator.

First let us recall that let $\tau(x)$ be the time of ruin with the initial reserve $x$: $\tau(x)=\inf\{t>0,X_{t}\leq0\}$. The survival probability $\Phi(x)$ is defined as follows:
\begin{eqnarray*}
\Phi(x)&=&P\{\tau(x)=\infty\}.
\end{eqnarray*}
As we know,   a general expression for $\Phi(x)$ does not exist, but the corresponding Laplace
transform of $\Phi(x)$ can be obtained  by \cite{Shimizu}, there we take $\omega\equiv1$ and $\delta=0$ which is the Laplace transform of our model:
\begin{eqnarray}\label{lp}
L_{\Phi}(s)&=&\int_{0}^{\infty}e^{-sx}\Phi(x)dx
\nonumber\\&=&\frac{1-\frac{\lambda\mu}{c}}{s+\frac{\sigma^{2}}{2c}s^{2}-\frac{\lambda}{c}(1-l_{F}(s))}
\end{eqnarray}
with the estimators of the parameters in \eqref{lp}, we define $\widetilde{L_{\Phi}(s)}$ as the estimator of $L_{\Phi}(s)$:
\begin{eqnarray}\label{lpe}
\widetilde{L_{\Phi}(s)}&=&\frac{1-\widetilde{\rho_{n}}^{*}}{s+\frac{\widetilde{\sigma_{n}^{2}}}{2c}s^{2}-\frac{\widetilde{\lambda}_{n}}{c}(1-\widetilde{l_{F}}_{n}^{*}(s))}.
\end{eqnarray}
In order to estimate the original functions $\Phi(x)$, we will apply a regularized inversion of the Laplace transform proposed by \cite{Chauveau} to \eqref{lpe}. The regularized inversion is defined as follows, which is available for any $L^{2}$ functions.
\begin{defn} \label{4}
Let $m>0$ be a constant. The regularized Laplace inversion $L_{m}^{-1}:L^{2}\rightarrow L^{2}$ is given by
$$L_{m}^{-1}g(t)=\frac{1}{\pi^{2}}\int_{0}^{\infty}\int_{0}^{\infty}\Psi_{m}(y)y^{-\frac{1}{2}}e^{-tvy}g(v)dvdy$$
for a function $g\in L^{2}$ and $t\in(0,\infty)$, where
$$\Psi_{m}(y)=\int_{0}^{a_{m}}cosh(\pi x)cos(x\log y)dx$$
and $a_{m}=\pi^{-1}cosh^{-1}(\pi m)>0$.
\end{defn}

\begin{remark}\label{remark2}
It is well known that the norm $L^{-1}$ is generally unbounded, which causes the ill-posedness of the Laplace inversion. However, $L_{m}^{-1}$ is bounded for each $m>0$, in particular,
$$\|L_{m}^{-1}\|\leq m\|L\|=\sqrt{\pi}m$$
see \cite{Chauveau}, equation $(3.7)$.
\end{remark}

By \eqref{lp} and \eqref{lpe}, it is obvious that $L_{\Phi}(s) \notin L^{2}(0,\infty)$ and $\widetilde{L_{\Phi}}(s) \notin L^{2}(0,\infty)$. The regularized laplace transform inversion $L_{m}^{-1}$ of Definition ~\ref{4} does not apply at once. In order to ensure they are in $L^{2}(0,\infty)$, these functions $L_{\Phi}(s)$ and $\widetilde{L_{\Phi}}(s)$ will be slightly modified.

For arbitrary fixed $\theta>0$, we define $\Phi_{\theta}(x)=e^{-\theta x}\Phi(x)$.
It is obvious that
$$L_{\Phi_{\theta}}(s)=L_{\Phi}(s+\theta),\quad\widetilde{L_{\Phi_{\theta}}(s)}=\widetilde{L_{\Phi}(s+\theta)}.$$

Let us define an estimaor of $\Phi(x)$: for given numbers $\theta>0$ and $m(n)$, we denote by
\begin{eqnarray}\label{gx}
\widetilde{\Phi_{n}}(x)&=&e^{\theta x}\widetilde{\Phi_{\theta,m(n)}}(x),
\end{eqnarray}
where $\widetilde{\Phi_{\theta,m(n)}}(x)=L_{m(n)}^{-1}\widetilde{L_{\Phi_{\theta}}}(s)$.

Now we shall present our important result which states a convergence in probability of the $ISE$ on compacts.
\begin{thm}\label{main}
Suppose the same assumption as in Lemma ~\ref{cite} and the net profit condition: $c>\lambda\mu$. Moreover, suppose that there exist a constant $K>0$ such that $0\leq \Phi'(x)=g(x)\leq K<\infty$. Then, for numbers $m(n)$ such that $m(n)=\sqrt{\frac{T_{n}}{\log T_{n}}}$ as $n\rightarrow\infty$ and for any constant $B>0$, we have
$$\|\widetilde{\Phi}_{m(n)}-\Phi\|^{2}_{B}=O_{P}((\log T_{n})^{-1})\quad(n\rightarrow\infty).$$
\end{thm}

\subsection{The proof of Theorem~\ref{main}}\label{03}
To prove Theorem~\ref{main}, we need the following lemma:

The Lemma ~\ref{lem2}, which is essentially obtained by the proof of Theorem 3.2 in \cite{Chauveau}, shows that $L_{n}^{-1}$ can be a Laplace inversion asymptotically in $ISE$ sense.
\begin{lem}\label{lem2}
Suppose that, for a function $f\in L^{2}$ with the derivative $f'$,
$\int_{0}^{\infty}[t(t^{\frac{1}{2}}f(t))']^{2}t^{-1}dt<\infty$, then
$$\|L_{n} ^{-1}L_{f}-f\|=O\left((\log n)^{-\frac{1}{2}}\right)\quad (n\rightarrow\infty).$$
\end{lem}

Now we will prove the Theorem~\ref{main}.
\begin{proof}
By \eqref{gx}, we have
\begin{eqnarray*}
\|\widetilde{\Phi}_{m(n)}-\Phi\|^{2}_{B}&\leq&e^{2\theta B}\|\widetilde{\Phi}_{\theta,m(n)}-\Phi_{\theta}\|^{2}_{B}\\&\leq&2e^{2\theta B}\{\|L_{m(n)}^{-1}\widetilde{L_{\Phi_{\theta}}}-L_{m(n)}^{-1}L_{\Phi_{\theta}}\|^{2}+\|\Phi_{\theta,m(n)}-\Phi_{\theta}\|^{2}\}
\\&=&2e^{2\theta B}\left[\|I_{1}+I_{2}\|\right].
\end{eqnarray*}
In order to deal with $I_{2}$, let us write $\Phi'_{\theta}=g_{\theta}$ and note that
\begin{eqnarray*}
&&\int_{0}^{\infty}[x(\sqrt{x}\Phi_{\theta}(x))']^{2}\frac{1}{x}dx\\&=&\int_{0}^{\infty}[x(\frac{1}{2\sqrt{x}}\Phi_{\theta}(x)+x\sqrt{x}g_{\theta}(x))]^{2}\frac{1}{x}dx
\\&\leq&\int_{0}^{\infty}2\frac{1}{x}[x(\frac{1}{2\sqrt{x}}\Phi_{\theta}(x)]^{2}+\int_{0}^{\infty}2\frac{1}{x}[x\sqrt{x}g_{\theta}(x)]^{2}dx
\\&=&\int_{0}^{\infty}\frac{1}{2}\Phi_{\theta}^{2}(x)dx+2\int_{0}^{\infty}x^{2}g_{\theta}^{2}(x)dx
\\&\leq&\int_{0}^{\infty}\frac{1}{2}e^{-2\theta x}dx+2\int_{0}^{\infty}x^{2}[g(x)e^{-\theta x}-\theta \Phi(x)e^{-\theta x}]^{2}dx
\\&\leq&\frac{1}{4\theta}+4\int_{0}^{\infty}x^{2}g^{2}(x)e^{-2\theta x}dx+4\theta^{2}\int_{0}^{\infty}\Phi^{2}(x)x^{2}e^{-2\theta x}dx
\\&\leq&\frac{1}{4\theta}+4(K^{2}+\theta^{2})\int_{0}^{\infty}x^{2}e^{-2\theta x}dx
\\&<&+\infty.
\end{eqnarray*}
Therefore, by the Lemma ~\ref{lem2}, we may conclude that
\begin{eqnarray}\label{bx}
\|\Phi_{\theta,m(n)}-\Phi_{\theta}\|^{2}=O(\frac{1}{\log m(n)}).
\end{eqnarray}
Next, we consider the formula $I_{1}$. By \eqref{lp} and \eqref{lpe}, we have
\begin{eqnarray}\label{bc}
&&\|\widetilde{L_{\Phi_{\theta}}}-L_{\theta}\|^{2}
\nonumber\\&\leq&\int_{0}^{\infty}|\frac{1-\widetilde{\rho_{n}}^{*}}{\widetilde{D}(s+\theta)}-\frac{1-\rho}{D(s+\theta)}|^{2}ds
\nonumber\\&=&\int_{0}^{\infty}|\frac{1-\widetilde{\rho_{n}}^{*}}{\widetilde{D}(s+\theta)}-\frac{1-\rho}{\widetilde{D}(s+\theta)}+\frac{1-\rho}{\widetilde{D}(s+\theta)}-\frac{1-\rho}{D(s+\theta)}|^{2}ds
\nonumber\\&\leq&\int_{0}^{\infty}2\left(\frac{(1-\rho)(\widetilde{D}(s+\theta)-D(s+\theta))}{\widetilde{D}(s+\theta)D(s+\theta)}\right)^{2}
+2\left(\frac{(\widetilde{\rho_{n}}^{*}-\rho)}{\widetilde{D}(s+\theta)}\right)^{2}ds.
\nonumber\\&\leq&2\int_{0}^{\infty}\frac{(\widetilde{D}(s+\theta)-D(s+\theta))^{2}}{(s+\theta)^{4}(1-\widetilde{\rho_{n}}^{*})^{2}}ds
+2\int_{0}^{\infty}(\frac{\widetilde{\rho_{n}}*-\rho}{1-\widetilde{\rho_{n}}^{*}})^{2}\frac{1}{(s+\theta)^{2}}ds
\nonumber\\&=&2\int_{0}^{T_{n}}\frac{(\widetilde{D}(s+\theta)-D(s+\theta))^{2}}{(s+\theta)^{4}(1-\widetilde{\rho_{n}}^{*})^{2}}ds
\nonumber\\&&+2\int_{T_{n}}^{\infty}\frac{(\widetilde{D}(s+\theta)-D(s+\theta))^{2}}{(s+\theta)^{4}(1-\widetilde{\rho_{n}}^{*})^{2}}ds
\nonumber\\&&+2\int_{0}^{\infty}(\frac{\widetilde{\rho_{n}}^{*}-\rho}{1-\widetilde{\rho_{n}}^{*}})^{2}\frac{1}{(s+\theta)^{2}}ds
\end{eqnarray}
and
\begin{eqnarray}\label{hh}
&&\widetilde{D_{n}}(s+\theta)-D(s+\theta)
\nonumber\\&=&\frac{(s+\theta)^{2}}{2c}(\widetilde{\sigma_{n}}^{2}-\sigma^{2})+\frac{1}{c}\left((\lambda-\widetilde{\lambda_{n}})+(\widetilde{\lambda_{n}}\widetilde{l_{F}}_{n}^{*}(s+\theta)-\lambda l_{F}(s+\theta))\right)
\nonumber\\&=&O_{P}(T_{n}^{-\frac{1}{2}})+\frac{1}{c}(\widetilde{\lambda_{n}}\widetilde{l_{F}}_{n}^{*}(s+\theta)-\lambda l_{F}(s+\theta)).
\end{eqnarray}
Note that we used here the fact that $\widetilde{\sigma_{n}}^{2}-\sigma^{2}=o_{P}(T_{n}^{-\frac{1}{2}})$.

Now we consider the second term in \eqref{bc}. By \eqref{hh}, we have
\begin{eqnarray*}
&&\int_{T_{n}}^{\infty}\frac{(\widetilde{D}(s+\theta)-D(s+\theta))^{2}}{(s+\theta)^{4}(1-\widetilde{\rho_{n}}^{*})^{2}}ds
\\&\leq& O_{P}(T_{n}^{-1})\int_{T_{n}}^{\infty}\frac{1}{(s+\theta)^{4}}ds +\frac{2}{c^{2}}\int_{T_{n}}^{\infty}\frac{(\widetilde{\lambda_{n}}\widetilde{l_{F}}_{n}^{*}(s+\theta)-\lambda l_{F}(s+\theta))^{2}}{(s+\theta)^{4}(1-\widetilde{\rho_{n}}^{*})^{2}}ds
\\&\leq& O_{P}(T_{n}^{-1})\int_{T_{n}}^{\infty}\frac{1}{(s+\theta)^{4}}ds+\frac{4}{c^{2}}\int_{T_{n}}^{\infty}\frac{\widetilde{\lambda_{n}}^{2}+\lambda ^{2}}{(s+\theta)^{4}(1-\widetilde{\rho_{n}}^{*})^{2}}ds
\\&=& O_{P}(T_{n}^{-1})\int_{T_{n}}^{\infty}\frac{1}{(s+\theta)^{4}}ds+O_{P}(1)\int_{T_{n}}^{\infty}\frac{1}{(s+\theta)^{4}}ds
\\&=&O_{P}(1)\int_{T_{n}}^{\infty}\frac{1}{(s+\theta)^{4}}ds
\\&=&O_{P}(\frac{1}{T_{n}^{3}}).
\end{eqnarray*}
By Lemma ~\ref{cite}, the first term of \eqref{bc} is
\begin{eqnarray*}
2\int_{0}^{T}\frac{(\widetilde{D}(s+\theta)-D(s+\theta))^{2}}{(s+\theta)^{4}(1-\widetilde{\rho_{n}}^{*})^{2}}ds
&=& O_{P}(T_{n}^{-1})
\end{eqnarray*}
and the last term is
\begin{eqnarray*}
2\int_{0}^{\infty}(\frac{\widetilde{\rho_{n}}^{*}-\rho}{1-\widetilde{\rho_{n}}^{*}})^{2}\frac{1}{(s+\theta)^{2}}ds
&=& O_{P}(T_{n}^{-1}).
\end{eqnarray*}

Therefore,
\begin{eqnarray}\label{zh}
\|L_{m(n)}^{-1}\|^{2}\|\widetilde{L_{\Phi_{\theta}}}-L_{\Phi_{\theta}}\|^{2}&=&O_{p}(\frac{m^{2}(n)}{T_{n}}).
\end{eqnarray}
Combining \eqref{bx} and \eqref{zh}, we have
\begin{eqnarray}\label{jl}
\|\widetilde{\Phi}_{m(n)}-\Phi\|^{2}_{B}=O_{p}(\frac{m^{2}(n)}{T_{n}})+O_{p}(\frac{1}{\log m(n)}).
\end{eqnarray}
With an optimal $m(n)=\sqrt{\frac{T_{n}}{\log T_{n}}}$ balanceing the the right hand two terms in \eqref{jl}, the order becomes $O_{P}((\log T_{n})^{-1})$.
\end{proof}
\begin{remark}\label{remark3}
The explicit integral expression
\begin{eqnarray*}
\widetilde{\Phi}_{m(n)}(u)&=&\frac{e^{u\theta}}{\pi^{2}}\int_{0}^{\infty}\int_{0}^{\infty}e^{-usy}\widetilde{L_{\Phi_{\theta}}}(s)\Psi_{m(n)}(y)y^{-\frac{1}{2}}dsdy
\end{eqnarray*}
where $\Psi_{m(n)}(y)=\int_{0}^{a_{m(n)}}\cosh(\pi x)\cos(x\log(y))dx$ and $a_{m(n)}=\pi^{-1}\cosh^{-1}(\pi m(n))>0$ and $m(n)=\sqrt{\frac{T_{n}}{\log T_{n}}}$.
\end{remark}

\section{Goodness-of-fit Test for $\gamma_i$}\label{GoF}
In theorem \ref{aympto hat  bar F} we have proved the asymptotic properties of $\hat{F}(u)$, but as we can see the variance depends on the unknown parameter $\lambda$, when we want to do the problem of goodness-of-fit test for the distribution of $\gamma_i$, we want to find the estimator with the distribution free just with the null hypothese: $\mathcal{H_0}: F(x)=F_0^(x)$. To achieve this goal, we need the following convergence:

\begin{equation}\label{no lambda estimator}
\sqrt{\sum_{i=1}^nI_{\{|\Delta_iX|>\vartheta(h_n)\}}}\left(\hat{F}_n(u)-F(u)\right)\xrightarrow{D}\mathbb{N}(0,F(u)(1-F(u))
\end{equation}
The proof is the same as in theorem \ref{aympto hat  bar F}.  Now we construct the Cram\'er-von Mises $W_n^2$ and Kolmogorov-Smirnov $D_n$ statistics are
$$
W_n^2=n\int_{-\infty}^{\infty}\left[\hat{F}_n(x)-F_0(x)\right]^2dF_0(x),\,\,\,\, D_n=\sup_x\left|\hat{F}_n(x)-F_0(x)\right|
$$
From the equation \eqref{no lambda estimator}, we can easily get the limit behavior of these statistics
$$
W_n^2\xrightarrow{d} \int_0^1W_0(s)^2ds,\,\,\,\,\,\,\,\,\, \sqrt{n}D_n\xrightarrow{d}\sup_{0\leq s \leq 1}|W_0(s)|
$$
where $\{W_0(s),\, 0\leq s\leq 1\}$ is a Brownian bridge, the same as the i.i.d case.

\end{document}